\documentclass[11pt]{article}
\usepackage[utf8]{inputenc}
\usepackage{amsfonts, amsmath}
\usepackage{amssymb}
\usepackage{amsthm,amsmath,amscd}
\usepackage{enumerate}
\usepackage{url}
\usepackage[margin=25pt,font=small,labelfont=bf]{caption}
\usepackage{subfig}
\usepackage[numbers,square,sort&compress]{natbib}
\usepackage[colorlinks=true]{hyperref}
\newcommand{\defn}[1]{\textcolor{blue}{\emph{#1}}}
\newcommand*{\doi}[1]{doi: \href{https://dx.doi.org/#1}{\urlstyle{rm}\nolinkurl{#1}}}
\newcommand*{\arxiv}[1]{arXiv:  \href{https://arxiv.org/abs/#1}{\urlstyle{rm}\nolinkurl{#1}}}
\let\oldproofname=\proofname
\renewcommand{\proofname}{\rm\bf{\oldproofname}}

\newcommand{\RR}{\mathbb R}

\newcommand{\bna}{\begin{eqnarray}}
\newcommand{\ena}{\end{eqnarray}}
\newcommand{\ba}{\begin{eqnarray*}}
\newcommand{\ea}{\end{eqnarray*}}
\newcommand{\bs}[1]{}

\newtheorem{theorem}{Theorem}[section]

\newtheorem{lemma}[theorem]{Lemma}
\newtheorem{proposition}[theorem]{Proposition}

\newtheorem{definition}[theorem]{Definition}

\def\0{{\bf 0}}

\let\oldv=\v
\def\v{{\bf v}}

\def\x{{\bf x}}

\usepackage[margin=1in]{geometry}

\begin{document}
\title{
    Proving the Existence of a GOR Without Probability}

\author{
    Steven J. Gortler
    \and
    Louis Theran}
\date{}
\maketitle

\begin{abstract}
    In this note, we provide a new proof that a $D$-connected graph $G$ on $n$
    vertices has a general position orthogonal representation in $\RR^{n-D}$.
    Our argument, while based on many of the concepts from the original proof
    due to Lovász, Saks and Schrijver, does not use the probabilistic method.
\end{abstract}
\section{Introduction} \label{section:introduction}

Let $G$ be a graph with $n$ vertices and let $D\le n$ be a given dimension.
A \defn{representation} of $G$ in $\RR^D$ is an assignment to each vertex $i$ a vector
$\v_i$  in $\RR^D$.
An \defn{orthogonal representation (OR)} of $G$ in $\RR^D$ is an assignment to each vertex $i$ a vector
$\v_i$  in $\RR^D$ so that if $i\neq j$ is not an edge of $G$, then $\v_i \perp \v_j$.
A \defn{general position orthogonal representation (GOR)} of $G$ in $\RR^D$ is an OR with the added condition
that the $\v_i$ are in general linear position.
This means that any $D$ of the vectors $\v_i$
are linearly independent.
When clear from the context, we will often just refer to a
GOR without stating the graph or dimension.
We will also use the term $OR$ to refer the set of
all ORs for $G$ in $\RR^D$, and likewise the term $GOR$.
In~\cite{Lovasz-Schrijver,Lovasz-Schrijver2}, Lovász, Saks, and Schrijver prove the following,
strikingly simple, characterization of graphs with a GOR in dimension $\RR^D$.
\begin{theorem}
    \label{thm:main}
    A graph $G$ with $n$ vertices
    has a GOR in $\RR^D$ if and only if it is $(n-D)$-connected.
\end{theorem}
It is straightforward to establish
that a graph which is not $(n-D)$-connected cannot have a GOR in
$\RR^D$.
The difficult direction
of the proof is to show that any $(n-D)$-connected graph does have
a GOR in $\RR^D$.  In \cite{Lovasz-Schrijver,Lovasz-Schrijver2}
this step is done using the probabilistic method.  The authors
define a randomized algorithm for constructing an OR in $\RR^D$, and
show that, on an $(n-D)$-connected graph, the algorithm produces a
GOR with probability one.

In this note, we provide a novel ``structural'' proof for the existence of a GOR that does
not use the probabilistic method.  Our proof has a similar overall plan as the
one in \cite{Lovasz-Schrijver}, and uses many of those ideas.  The
novelty is that we substitute
probability statements with  topological ones.
Our proof is non-constructive, but it can be
used to justify the randomized algorithm described
in~\cite{Lovasz-Schrijver}.

A semi-algebraic set is \defn{irreducible} if its Zariski-closure is irreducible as an algebraic set.
In~\cite{Lovasz-Schrijver}, the authors follow up their main theorem with the following result.
\begin{theorem}
    \label{thm:irr1}
    $GOR$, if non-empty, is an irreducible semi-algebraic set.
\end{theorem}
Our approach is to take this result, which can be proven independently of
Theorem \ref{thm:main}, as the starting point.

\section{The \texorpdfstring{$GOR^+_\sigma$}{GOR plus sigma} construction}
To prove Theorem~\ref{thm:irr1}, we analyze the following process from
\cite[Proof of Theorem 2.1]{Lovasz-Schrijver}.
Fix an ordering $\sigma = \sigma_1\sigma_2\cdots\sigma_n$ of the vertices of $G$,
which will be the set $[n] = \{1, 2, \ldots, n\}$.
To make the description simpler,
we take, for the moment, $\sigma = 123\cdots (n-1)n$ to be the usual ordering on
$[n]$.  We construct a representation of $G$ inductively as follows.  For each $i=1, \ldots, n$:
\begin{itemize}
    \item Find the non-neighbors $1 \le i_1 \le i_2 \le \cdots \le i_k < i$ of $i$
          that come before it.
          (Note that there will be at most $D-1$ of these when $G$ is $(n-D)$-connected.)
    \item If the vectors $\v_{i_1}, \ldots, \v_{i_k}$ are linearly independent,
          place $\v_i$ anywhere in the $(D-k)$-dimensional subspace orthogonal to the linear span
          of these preceding non-neighboring vectors $\v_{i_j}$.
          Otherwise, set $\v_i = 0$.
\end{itemize}
This construction always produces an OR of $G$.

\begin{definition}
    Fix $G$ and $D$. For $i=1,\ldots, n$,
    we define $GOR^+_{\sigma,i}$
    to be the set of possible outputs of the above construction after $i$ steps.
    By convention, we define $GOR^+_{\sigma,0}$ to be $\{\bf 0\}$.
    We then define $GOR^+_{\sigma}$ to be $GOR^+_{\sigma,n}$,
    the set of possible outputs after all $n$ steps.
\end{definition}
The motivation for the construction above, as explained in
\cite[Page 447]{Lovasz-Schrijver}, is that
the set $GOR^+_{\sigma}$ is the image of a polynomial map.
At each stage $i$ of the construction, we have the, already chosen, vectors $\v_1, \ldots, \v_{i-1}$,
forming a point in $GOR^+_{\sigma,i-1}$.
Lovász--Saks--Schrijver describe
a (determinant-based) polynomial map
\[
    \phi_{i}: GOR^+_{\sigma,i-1} \times \RR^{N_{i}} \rightarrow \RR^{iD}
\]
with the property that:
\begin{itemize}
    \item The input in the $\RR^{N_i}$ factor are $N_i$ free parameters.
    \item The output consists of the input vectors $\v_1, \ldots, \v_{i-1}$
          from the first factor, together with
          a placement of $\v_i$ in the orthogonal complement of the span
          of the preceding non-neighbors $\v_{i_1}, \ldots, \v_{i_k}$ of the
          vertex $i$.
    \item The output for $\v_i$ is always the zero vector
          if $\v_{i_1}, \ldots, \v_{i_k}$
          are linearly dependent. Otherwise, by varying the second factor,
          every vector orthogonal to the span of $\v_{i_1}, \ldots, \v_{i_k}$ can
          be obtained.
\end{itemize}
Composing all the stages (noting that $GOR^+_{\sigma,0} = \{{\bf 0}\}$,
so we can get started), we obtain,  a map:
\[
    \phi := \phi_n \circ \cdots \circ \phi_1 : \RR^N \to \RR^{nD}
\]
where $N$ is the total number of parameters we needed.  The upshot is that
\[
    GOR^+_\sigma = \phi(\RR^N)
\]

Since $\RR^N$ is irreducible and $\phi$ polynomial,
it is immediate that:
\begin{lemma}
    \label{lem:irr2}
    $GOR^+_{\sigma}$ is an irreducible semi-algebraic set.
\end{lemma}
In fact,  we can say more.
\begin{lemma}
    \label{lem:dense}
    If $U$ is a non-empty Zariski-open subset of $GOR^+_{\sigma}$, then
    $U$ is dense  in  the standard topology on $GOR^+_{\sigma}$.
\end{lemma}
This is a general, and relatively standard, fact about parameterized semi-algebraic
sets.  We give a proof for completeness.
\begin{proof}
    The complement of $U$
    in $GOR^+_{\sigma}$
    lies in the vanishing set of some polynomial $f$ that
    does not vanish identically on $GOR^+_{\sigma}$.  Pulling back $f$, we obtain a polynomial
    $f\circ \phi$ that does not vanish identically on $\RR^N$ . Let $V$ be the vanishing
    set of $f\circ \phi$.  As a proper algebraic subset of $\RR^N$, $V$ is nowhere dense
    in the standard topology on $\RR^N$.  Hence, $\phi(\RR^N\setminus V) \subseteq U$ is
    dense in the standard topology on  $GOR^+_{\sigma}$.
\end{proof}
For any ordering $\sigma$, we will have
\[
    GOR \subsetneq  GOR^+_{\sigma} \subseteq OR
\]
We can start from a GOR and play the process in reverse,
by forgetting the vertices one at a time, which gives the first
inclusion.
Meanwhile, every $GOR^+_\sigma$ is an OR, giving the second inclusion.
The first inclusion is strict,
since $GOR^+_{\sigma}$ contains ORs with some of the $\v_i$ equal to the
zero vector. The second inclusion is also strict as long as $G$
is not complete (a representation $\v$ in  $GOR^+_{\sigma}$ cannot have a non-zero
$\v_i$ with a previous non-neighbor set to ${\bf 0}$.)
In the rest of this note, we will deduce that, assuming appropriate
connectivity, $GOR$ is non-empty.

\section{Orderings}

Now we want to compare $GOR^+_{\sigma}$ to $GOR^+_{\tau}$, where $\tau$ is some
other ordering of $[n]$.

\begin{definition}
    We say that two subsets $A$ and $B$ of a topological space
    are \defn{almost the same (ATS)} if $A \cap B$ is dense in both $A$ and $B$.
\end{definition}
We will frequently use the fact that being ATS is an equivalence relation
on semi-algebraic sets, which we state here and prove in the appendix.
\begin{lemma}\label{lem: ats eqvrel}
    Let $A,B,C\subseteq \RR^D$ be semi-algebraic sets.  If $A$ and $B$ are ATS and $B$ and $C$
    are ATS, then $A$ and $C$ are ATS.
\end{lemma}

The main proposition in this section is:
\begin{proposition}
    \label{prop:order}
    If $G$ is $(n-D)$-connected, and $\sigma$ and $\tau$ are two vertex orderings,
    then $GOR^+_{\sigma}$  and $GOR^+_{\tau}$ are ATS.
\end{proposition}
The rest of the section will be occupied with the proof, which is an induction on the statement
\begin{equation}\label{eq: IH} \tag{$\star$}
    \text{{\it For all $\sigma$ and $\tau$ such that, $\sigma_j = \tau_j = j$ for all $i < j\le n$,
    $GOR^+_{\sigma}$  and $GOR^+_{\tau}$ are ATS.}
    }
\end{equation}

\subsection{Base case 1}
To simplify things, let us first assume that we have a graph $G^-$ with $D$ vertices, and a
vertex ordering $\sigma$, and we are looking for GORs in $\RR^D$.
\begin{lemma}
    \label{lem:hasGOR}
    For $G^-$, a graph with $D$ vertices,  $GOR$ is dense in $GOR^+_\sigma$.
\end{lemma}
\begin{proof}
    $GOR$ is a Zariski open  subset of $GOR^+_{\sigma}$.  Since we only have $D$ vertices, there must
    be a GOR in $\RR^D$ (as we can just use the elementary vectors), and thus $GOR$ is non-empty. Thus from
    Lemma~\ref{lem:dense} it is a dense subset.
\end{proof}
Since the above lemma did not depend
on the ordering, using $GOR$ as the common dense subset, we get:
\begin{lemma}
    \label{lem:ats1}
    For $G^-$, a graph with $D$ vertices, and $\sigma$ and $\tau$ two vertex orderings,
    $GOR^+_\sigma$ and $GOR^+_\tau$ are ATS.
\end{lemma}

Now we want to upgrade this to a graph $G$ with $n$ vertices, giving us
our first base case.
\begin{lemma}\label{lem: base 1}
    Let $G$ be a graph with $n$ vertices, and $\sigma$ and $\tau$ two vertex orderings such that
    \[
        \sigma_j = \tau_j\qquad \text{for all $D < j \le n$}
    \]
    Then $GOR^+_\sigma$ and $GOR^+_\tau$ are ATS.
\end{lemma}
\begin{proof}
    Let $G^-$ be the subgraph of $G$ induced by the first $D$ vertices in the orderings $\sigma$ and $\tau$.
    This gives us the two sets
    $GOR^+_\sigma(G^-)$ and $GOR^+_\tau(G^-)$ that are ATS from Lemma~\ref{lem:ats1}.

    Next we move back to the original graph $G$.
    Consider the polynomial map
    $\psi: \RR^{D^2} \times \RR^{N} \rightarrow \RR^{nD}$
    defined by
    \[
        \psi = \phi_{\sigma_n} \circ \cdots \circ \phi_{\sigma_{D+1}} =  \phi_{\tau_n} \circ \cdots \circ \phi_{\tau_{D+1}}
    \]
    where we have used that $\sigma$ and $\tau$ are the same on their last $n-D$ elements.
    We interpret the $\RR^{D^2}$ factor as a representation of the $D$ vertices of $G$
    that come first in the orderings $\sigma$ and $\tau$,
    and the $\RR^N$ factor as the remaining free variables corresponding to free parameters.

    The image  $\psi(GOR^+_\sigma(G^-) \times \RR^{N})$ is
    $GOR^+_\sigma(G)$,
    while
    the image  $\psi(GOR^+_{\tau}(G^-) \times \RR^{N})$ is
    $GOR^+_\tau(G)$.
    Since $GOR^+_\sigma(G^-)$ and $GOR^+_{\tau}(G^-)$  are ATS, so too are
    $GOR^+_\sigma(G^-) \times \RR^N$ and
    $GOR^+_{\tau}(G^-) \times \RR^N$.
    Thus,
    \[
        \psi((GOR^+_\sigma(G^-) \times \RR^N)\cap (GOR^+_{\tau}(G^-) \times \RR^N))
    \]
    is dense in both images $GOR^+_\sigma(G)$ and
    $GOR^+_\tau(G)$ under the continuous map $\psi$, making them ATS.
\end{proof}

\subsection{Base Case 2}
The second base case of interest is the following
\begin{lemma}
    \label{lem:edgeCon}
    Suppose that $\sigma$ and $\tau$ are two orderings on  $[n]$ so that
    \[
        \sigma_j = \tau_j\quad \text{for all $j\in [n] \setminus \{i,i+1\}$} \qquad \sigma_{i+1} = \tau_i \qquad \sigma_i = \tau_{i+1}
    \]
    and that $\{\sigma_i,\sigma_{i+1}\}$ is an edge of $G$.  Then $GOR^+_\sigma = GOR^+_\tau$.
\end{lemma}
\begin{proof}
    When we run the process using either $\sigma$
    or $\tau$, the constraints on the placements of the $i$th and $i+1$st
    vertices
    will be independent
    of the placement of the other, and the
    rest of the vertex ordering is shared by $\sigma$ and $\tau$.
\end{proof}

\subsection{Inner induction}
We start with a special case of the inductive step.  It is here we use
the connectivity hypothesis.
\begin{lemma}
    \label{lem:con}
    If $G$ is $(n-D)$-connected and $i \ge D$, then  $\sigma_i$ and $\sigma_{i+1}$
    are connected by a path that visits only vertices in  $\{\sigma_1, \ldots, \sigma_{i+1}\}$.
\end{lemma}
\begin{proof}
    If every path from $\sigma_i$ to $\sigma_{i+1}$ visits some $\sigma_{j}$ with $j > i + 1$,
    then $\{\sigma_{i+2}, \ldots, \sigma_n\}$ is a cut set.  Since $i \ge D$, this would contradict
    $(n-D)$-connectivity.
\end{proof}
Now we can deal with a special case of the inductive step.
\begin{lemma}
    \label{lem:later}
    Suppose that $G$ is $(n-D)$ connected. Let $D\le i < n$ be fixed and suppose that the
    inductive hypothesis \eqref{eq: IH} holds for $i$.  If
    \[
        \sigma_j = \tau_j\quad \text{for all $j\in [n] \setminus \{i,i+1\}$} \qquad \sigma_{i+1} = \tau_i \qquad \sigma_i = \tau_{i+1}
    \]

    then $GOR^+_\sigma$ and $GOR^+_\tau$ are ATS.
\end{lemma}
\begin{proof}
    By Lemma \ref{lem:con}, there is a path of length
    $m+1$ for some $m\ge 0$:
    \[
        \sigma_i, \sigma_{j_1}, \ldots, \sigma_{j_m}, \sigma_{i+1}
    \]
    in $G$, where for each $\ell$, $1\le j_\ell\le i-1$.  The proof is by induction
    on $k$, which is an upper bound on the
    number of interior vertices
    in the shortest path
    between the $i$th and $(i+1)$st vertex in an ordering.  (These
    path endpoints are
    $\sigma_i$ and $\sigma_{i+1}$ in $\sigma$, but will be other
    vertices as we consider different orderings in the proof.)

    The base case of $k=0$ is Lemma \ref{lem:edgeCon}.
    Now suppose that the statement holds for some $k$ and assume that
    $\sigma_i$ is connected to $\sigma_{i+1}$ by a shortest path
    \[
        \sigma_i, \sigma_{j_1}, \ldots, \sigma_{j_{k+1}}, \sigma_{i+1}
    \]
    with $1\le j_\ell \le i-1$ for each $\ell$.  We now
    proceed as in \cite[Page 444]{Lovasz-Schrijver}, by making exchanges
    that preserve ATS either by the ``outer'' IH \eqref{eq: IH} or the ``inner'' IH on $k$.
    For completeness, the steps are:
    \begin{itemize}
        \item Exchange $\sigma_i$ with $\sigma_{j_1}$ to get
              an ordering $\sigma^1 = \cdots \sigma_i \cdots \sigma_{j_1}\sigma_{i+1}\cdots\sigma_n$.
              This is an exchange of ``early'' positions beween $1$ and $i$, so the
              IH \eqref{eq: IH} applied to $\sigma$ and $\sigma_1$ implies that $GOR^+_\sigma$ is ATS as $GOR^+_{\sigma^1}$.
        \item Exchange $\sigma_{j_1}$ with $\sigma_{i+1}$ in $\sigma^1$ to get
              an ordering $\sigma^2 = \cdots \sigma_i \cdots \sigma_{i+1}\sigma_{j_1}\cdots\sigma_n $.
              Note that $\sigma_{j_1}$ is connected to $\sigma_{i+1}$ by a path of length at most
              $k$. The IH on $k$ applies to $\sigma^1$ and $\sigma^2$, so $GOR^+_{\sigma^1}$ is  ATS as $GOR^+_{\sigma^2}$.
        \item Exchange $\sigma_{i}$ with $\sigma_{i+1}$ in $\sigma^2$ to get
              an ordering $\sigma^3 = \cdots \sigma_{i+1} \cdots \sigma_{i}\sigma_{j_1}\cdots\sigma_n$.  As in the first step, this is an
              ``early'' exchange, so the IH \eqref{eq: IH} used on $\sigma^2$ and $\sigma^3$
              implies that $GOR^+_{\sigma^2}$ is ATS as $GOR^+_{\sigma^3}$
        \item Exchange $\sigma_{i}$ with $\sigma_{j_1}$ in $\sigma^3$ to get
              an ordering $\sigma^4 = \cdots \sigma_{i+1} \cdots \sigma_{j_1}\sigma_{i}\cdots\sigma_n$.  We know that
              $\sigma_i$ and $\sigma_{j_1}$ are connected by an edge, so the $k=0$ base case used on $\sigma^3$ and $\sigma^4$
              implies that $GOR^+_{\sigma^3}$ is ATS as $GOR^+_{\sigma^4}$.
        \item  Exchange $\sigma_{i+1}$ with $\sigma_{j_1}$ in $\sigma^4$ to get
              the ordering $\tau$.  As in the first and third steps, the IH \eqref{eq: IH} applied to $\sigma^4$ and $\tau$
              implies that $GOR^+_{\sigma^4}$ and  $GOR^+_\tau$ are ATS.
    \end{itemize}
    Hence, $GOR^+_\sigma$ and $GOR^+_\tau$ are ATS by Lemma \ref{lem: ats eqvrel}.
\end{proof}

\subsection{Outer induction}
By Lemma \ref{lem: base 1},
we may assume that the inductive hypothesis \eqref{eq: IH} holds for
$i = D$.  Let us now suppose that \eqref{eq: IH} holds for some $i$ with $D \le i < n$ and that
and that $\sigma$ and $\tau$ are orderings so that
\[
    \sigma_j = \tau_j\qquad \text{for all $i+1 < j \le n$}
\]
We will show that $GOR^+_\sigma$ and $GOR^+_\tau$ are ATS to close the
induction.  If $\tau_{i+1} = \sigma_{i+1}$, we are done by the
inductive hypothesis. Otherwise,
we next show how to reduce to
the case of Lemma~\ref{lem:later},
by using
the inductive hypothesis \eqref{eq: IH},
which lets us
rearrange the first $i$ numbers in the
orderings.

At this point, we are assuming  that $\tau_{i+1} \neq \sigma_{i+1}$.  This means that
we have
\[
    \sigma = \cdots \sigma_{k-1}\tau_{i+1}\sigma_{k+1}\cdots \sigma_i
    \sigma_{i+1}
    \sigma_{i+2}\cdots \sigma_n
    \qquad
    \text{and}
    \qquad
    \tau = \cdots \tau_{j-1}\sigma_{i+1}\tau_{j+1}\cdots \tau_i \tau_{i+1}\tau_{i+2}\cdots \tau_n
\]
for some $1\le j,k\le i$.  Since we assume that $\sigma$ and $\tau$ are equal in positions
after $i+1$, we have
\[
    \tau = \cdots \tau_{j-1}\sigma_{i+1}\tau_{j+1}\cdots \tau_i \tau_{i+1}\sigma_{i+2}\cdots \sigma_n,
\]
and that, as sets
\[
    \{\sigma_1, \ldots, \sigma_{i+1}\} = \{\tau_1, \ldots, \tau_{i+1}\}.
\]
By the IH \eqref{eq: IH}, $GOR^+_\sigma$ is ATS as
$GOR^+_{\sigma'}$ and $GOR^+_\tau$ is ATS as  $GOR^+_{\tau'}$,
where
\[
    \sigma' = \sigma_1 \cdots \sigma_{k-1}\sigma_{k+1}\cdots \sigma_i\tau_{i+1}\sigma_{i+1}\sigma_{i+2}\cdots \sigma_n
    \qquad
    \text{and}
    \qquad
    \tau' = \sigma_1 \cdots \sigma_{k-1}\sigma_{k+1}\cdots \sigma_i\sigma_{i+1}\tau_{i+1}\sigma_{i+2}\cdots \sigma_n
\]
because these can be obtained from $\sigma$ and $\tau$ respectively by rearranging the first $i$
numbers of the ordering.  Lemma \ref{lem:later} then implies that
$GOR^+_{\sigma'}$ and $GOR^+_{\tau'}$ are ATS.  An application of Lemma \ref{lem: ats eqvrel}
closes the induction and thus proves
Proposition~\ref{prop:order}.

\section{Obtaining a GOR}
We are now in a position to complete the proof of Theorem~\ref{thm:main}.
\begin{definition}
    Let $I$ be a subset of $[n]$ of cardinality $D$.
    Let $GP(I)$
    be the
    Zariski-open subset of of $\RR^{nD}$ where the vertices of $I$ are placed in general position.
\end{definition}

\begin{lemma}
    \label{lem:nonempt1}
    Let $I\subseteq [n]$ be a fixed subset of cardinality $D$, and
    $\tau$ an ordering of $[n]$ so that $I = \{\tau_1, \ldots, \tau_D\}$.
    Then $GOR^+_\tau \cap GP(I)$ is non-empty.
\end{lemma}
\begin{proof}
    When we are running the $GOR^+_\tau$ construction
    process in $\RR^D$, each of the vertices $\tau_i$ in
    $I$ has fewer than $D$ preceding neighbors.  It follows that the process can select
    $\v_{\tau_1}, \ldots, \v_{\tau_D}$ in general position.  Whatever happens later
    in the process, these vectors will remain in general position, so $GOR^+_\tau \cap GP(I)$ is non-empty.
\end{proof}
\begin{lemma}
    \label{lem:nonempt2}
    Let $G$ be an $(n-D)$-connected graph with $n$ vertices and
    $I\subseteq [n]$ be a fixed subset of cardinality $D$.  For
    any ordering $\sigma$ of $[n]$, $GOR^+_\sigma \cap GP(I)$ is
    open and dense in $GOR^+_\sigma$.
\end{lemma}
\begin{proof}
    From Lemma~\ref{lem:nonempt1}, $GOR^+_\tau \cap GP(I)$ is non-empty
    for some ordering $\tau$ of $[n]$.  Since $GOR^+_\tau \cap GP(I)$
    is Zariski-open in $GOR^+_\tau$ it is open in the standard topology
    on  $GOR^+_\tau$.   Proposition~\ref{prop:order}, implies that
    $GOR^+_\sigma \cap GOR^+_\tau$ is dense in $GOR^+_\tau$;
    in particular, it must meet the non-empty open set $GOR^+_\tau \cap GP(I)$.
    Hence, $GOR^+_\sigma  \cap GP(I)$
    is non-empty.
    As a Zariski-open subset of $GOR^+_\sigma$,
    $GOR^+_\sigma  \cap GP(I)$
    is open in the standard topology, and
    from Lemma~\ref{lem:dense}, it is dense
    in $GOR^+_\sigma$.
\end{proof}

The deduction of Theorem \ref{thm:main} is now formulaic.  Fix
any ordering $\sigma$ of $[n]$.  By Lemma \ref{lem:nonempt2}, each
set in the intersection
\[
    \bigcap_I (GOR^+_\sigma\cap GP(I))\qquad \text{($I\subseteq [n]$ and $|I| = D$)}
\]
is open and dense in $GOR^+_\sigma$.  A finite intersection of such sets
will also be open and dense.  We conclude that $GOR\subseteq GOR^+_\sigma$ is
open and dense.  In particular, it is not empty.
This completes the proof of the Theorem~\ref{thm:main}.

We note that density of $GOR$ in
$GOR^+_\sigma$
is also observed in \cite[Equation 10.32]{lovbook}
as a consequence of Theorem~\ref{thm:main}.

\section{Comparison with Lovász--Saks--Schrijver}
Here we compare and contrast our new proof with
the original in~\cite{Lovasz-Schrijver,Lovasz-Schrijver2}.  The authors
consider, for each fixed ordering $\sigma$ of $[n]$ a randomized process.
For simplicity, we describe it for $\sigma = 12\cdots (n-1)n$.  The
input $G$ is a graph on $n$ vertices that is $(n-D)$-connected.
For each $i=1, \ldots, n$:
\begin{itemize}
    \item Find the non-neighbors $1 \le i_1 \le i_2 \le \cdots \le i_k < i$ of $i$
          that come before it.
    \item Select $\v_i$ uniformly at random among unit-length vectors orthogonal to
          the linear span of $\v_{i_1},\ldots, \v_{i_k}$.  (A slight variation for controlling the
          random choice is described in~\cite{lovbook}.)
\end{itemize}
This process produces a probability distribution $P_\sigma(\v)$ over $OR$,
which is, in general, a reducible algebraic set. In this
process, the dimension of the orthogonal space
used for placing $\v_i$
can vary based on the positions of the preceding non-neighboring vectors.
As such, it is
not a randomized version of the $GOR^+_\sigma$ generation process
described above.

The probabilistic counterpart to being ATS employed in \cite{Lovasz-Schrijver,Lovasz-Schrijver2}
relates the zero probability events of two distributions.
\begin{definition}\label{def: mac}
    Two probability distributions $\mu$ and $\nu$ over $\RR^N$ are called
    \defn{mutually absolutely continuous (MAC)} if they have the same null
    sets.
\end{definition}
The main step in \cite{Lovasz-Schrijver,Lovasz-Schrijver2} is to establish:
\begin{proposition}
    \label{prop:orderLSS}
    If $G$ is $(n-D)$-connected, and $\sigma$ and $\tau$ are two vertex orderings,
    then $P_{\sigma}$  and $P_{\tau}$ are MAC.
\end{proposition}
This proposition is similar in form to our Proposition~\ref{prop:order},
and, indeed, the double inductive proof above is structurally similar to
the one of Proposition~\ref{prop:orderLSS} in \cite{Lovasz-Schrijver,Lovasz-Schrijver2}.
The key difference is that we have replaced statements about null sets of
distributions $P_\sigma$ over $OR$ with topological statements about the irreducible sets
$GOR^+_\sigma$.

\def\v{\oldv}

\newpage
\appendix

\section{Proof of Lemma \ref{lem: ats eqvrel}}

Recall that a semi-algebraic set
$S \in \RR^D$ can be \defn{stratified}\cite{RoyAlg}.
This means that $S$ can be written as
the
disjoint union of a finite number of smooth submanifolds
$S_i \in \RR^D$ (of various dimensions)
called
the \defn{strata} of this stratification.
A stratum  $S_i$ of dimension $d_i$ has the property that its
closure in $S$ is the union of
some other strata $S_j$ of strictly
lower dimension.
The \defn{dimension} of $S$ is defined as the
maximum dimension of its smooth strata
in one (equiv. any)
stratification.

\begin{lemma}
    \label{lem:denseOpen}
    Suppose  $A$ and $B$ are semi-algebraic sets
    with $A$ a dense subset of $B$. Then there exists
    a
    set $U \subseteq A$ that is open and dense in $B$.
\end{lemma}
\begin{proof}

    We first stratify $B$.
    We then define $B'$ as the union of the
    strata of that are not in the closure of
    any other stratum (necessarily of higher dimension).
    We denote the strata comprising $B'$ as $B'_i$,
    each with its dimension denoted as $d_i$.
    Since we only removed strata in the closure of
    maintained strata,
    $B'$ is dense in $B$.

    Let $\x$ be a point in $B'_i$.
    Suppose $\x$
    was in the closure of $B\backslash B'_i$, then
    $B'_i$, as the stratum containing $\x$, would
    have been removed.
    Thus $\x$ must have a
    neighborhood in $B$ that
    is contained in $B'_i$.
    Thus $B'_i$ is open in $B$.

    We next define $A'_i := A \cap B'_i$
    and we define $A' :=\cup_i A'_i$.
    Since $A$ is dense in $B$, and $B'_i$ is open
    in $B$, $A'_i$ is dense in $B'_i$.
    It follows that $A'_i$ must be of dimension $d_i$, otherwise it could
    not be dense in $B'_i$.

    Next we stratify (each) $A'_i$.
    We define $A''_i$ to be the union of the $d_i$ (top)
    dimensional strata comprising $A'_i$.
    We denote the strata comprising $A''_i$  as $A''_{ij}$.
    $A''_i$ must be dense in $B'_i$, as we have
    only removed
    lower dimensional manifolds, which can be nowhere
    dense in $B'_i$.
    $A''_{ij}$, is a $d_i$-dimensional smooth manifold included in $B'_i$, a smooth manifold of the same dimension.
    Thus,
    $A''_{ij}$
    must
    be an open subset of $B'_i$.
    Finally we define $U:=\cup_i A''_i$ giving us
    a set that
    must be open and dense in $B'$ and thus also in $B$.
\end{proof}

\begin{lemma}
    \label{lem:sadense}
    Suppose that $A$, $B$, and $C$ are semi-algebraic
    sets, with $A$ and $B$  dense subsets of $C$.
    Then $A \cap B$ is dense in $C$.
\end{lemma}
\begin{proof}
    From Lemma~\ref{lem:denseOpen}, there exist
    $U\subseteq A$ and $V\subseteq B$
    that are both open and dense in $C$. Thus
    $U \cap V$ is dense in $C$. Thus $A \cap B$
    is dense in $C$.
\end{proof}

\begin{proof}[Proof of Lemma \ref{lem: ats eqvrel}]
    By assumption, both $A \cap B$ and $B \cap C$
    are dense subsets of $B$. From Lemma~\ref{lem:sadense},
    their intersection $A \cap B \cap C$ is a
    dense subset of $B$. Thus $A \cap B \cap C$
    is dense in both $A\cap B$ and $B \cap C$.
    By assumption, $A \cap B$ is dense in $A$ thus $A \cap B \cap C$
    is dense in  $A$.
    Similarly, by assumption, $B \cap C$ is dense in $C$ thus $A \cap B \cap C$
    is dense in  $C$. Thus $A \cap C$ is dense in
    both $A$ and $C$.
\end{proof}


\begin{thebibliography}{4}
    \providecommand{\natexlab}[1]{#1}
    \providecommand{\url}[1]{\texttt{#1}}
    \expandafter\ifx\csname urlstyle\endcsname\relax
        \providecommand{\doi}[1]{doi: #1}\else
        \providecommand{\doi}{doi: \begingroup \urlstyle{rm}\Url}\fi

    \bibitem[Basu et~al.(2006)Basu, Pollack, and Roy]{RoyAlg}
    S.~Basu, R.~Pollack, and M.-F. Roy.
    \newblock \emph{Algorithms in real algebraic geometry}, volume~10 of \emph{Algorithms and Computation in Mathematics}.
    \newblock Springer-Verlag, Berlin, second edition, 2006.
    \newblock \doi{10.1007/3-540-33099-2}.

    \bibitem[Lov\'{a}sz(2019)]{lovbook}
    L.~Lov\'{a}sz.
    \newblock \emph{Graphs and geometry}, volume~65 of \emph{American Mathematical Society Colloquium Publications}.
    \newblock American Mathematical Society, Providence, RI, 2019.
    \newblock \doi{10.1090/coll/065}.

    \bibitem[Lov{\'a}sz et~al.(1989)Lov{\'a}sz, Saks, and Schrijver]{Lovasz-Schrijver}
    L.~Lov{\'a}sz, M.~Saks, and A.~Schrijver.
    \newblock Orthogonal representations and connectivity of graphs.
    \newblock \emph{Linear Algebra Appl.}, 114/115:\penalty0 439--454, 1989.
    \newblock \doi{10.1016/0024-3795(89)90475-8}.

    \bibitem[Lov\'asz et~al.(2000)Lov\'asz, Saks, and Schrijver]{Lovasz-Schrijver2}
    L.~Lov\'asz, M.~Saks, and A.~Schrijver.
    \newblock A correction: ``{O}rthogonal representations and connectivity of graphs''.
    \newblock \emph{Linear Algebra Appl.}, 313\penalty0 (1-3):\penalty0 101--105, 2000.
    \newblock \doi{10.1016/S0024-3795(00)00091-4}.

\end{thebibliography}
\end{document}